\documentclass[11pt]{article}
\usepackage{amssymb,url,amsmath,graphicx, amsthm,mathdots}
\usepackage{accents}
\usepackage{tikz}
\usetikzlibrary{positioning, calc, chains}
\usetikzlibrary{shapes.misc}
\usetikzlibrary{shapes.symbols}
\usetikzlibrary{shapes.geometric}
\usetikzlibrary{shapes.arrows}
\usetikzlibrary{fit}
\usetikzlibrary{shadows}

\usepackage{pgfplots}
\usepackage{pgfplotstable}

\usepackage{subcaption}

\newtheorem{theorem}{Theorem}[section]

\newtheorem{proposition}{Proposition}[theorem]
\newtheorem{lemma}{Lemma}[theorem]

\newcommand{\citep}[1]{\cite{#1}}
\newcommand{\bfc}{\mathbf{c}}
\newcommand{\bfd}{\mathbf{d}}
\newcommand{\bfb}{\mathbf{b}}
\newcommand{\bfp}{\mathbf{p}}
\newcommand{\bfq}{\mathbf{q}}

\newcommand{\bfy}{\mathbf{y}}
\newcommand{\bfx}{\mathbf{x}}
\newcommand{\bfe}{\mathbf{e}}
\newcommand{\bfpi}{\mathbf{\Pi}}

\usepackage{algorithmic}
\usepackage{algorithm}

\DeclareMathOperator{\Bez}{Bez}
\DeclareMathOperator{\diag}{diag}

\title{Structured inversion of the Bernstein mass matrix}
\date{}
\author{Larry Allen\thanks{Department of Mathematics, Baylor
    University; One Bear Place \#97328; Waco, TX 76798-7328.
    Email: larry\_alllen@baylor.edu.}
    \and
    Robert C.~Kirby\thanks{Department of Mathematics, Baylor
    University; One Bear Place \#97328; Waco, TX 76798-7328.
    Email: robert\_kirby@baylor.edu.  This work was supported by NSF grant 1525697.}}
\begin{document}
\maketitle

\begin{abstract}
Bernstein polynomials, long a staple of approximation theory and computational geometry, have also increasingly become of interest in finite element methods.  Many fundamental problems in interpolation and approximation give rise to interesting linear algebra questions.  In~\cite{kirby2017fast}, we gave block-structured algorithms for inverting the Bernstein mass matrix on simplicial cells, but did not study fast alorithms for the univariate case.  Here, we give several approaches to inverting the univariate mass matrix based on exact formulae for the inverse; decompositions of the inverse in terms of Hankel, Toeplitz, and diagonal matrices; and a spectral decomposition.  In particular, the eigendecomposition can be explicitly constructed in $\mathcal{O}(n^2)$ operations, while its accuracy for solving linear systems is comparable to that of the Cholesky decomposition.  Moreover, we study conditioning and accuracy of these methods from the standpoint of the effect of roundoff error in the $L^2$ norm on polynomials, showing that the conditioning in this case is far less extreme than in the standard 2-norm.
\end{abstract}

\section{Introduction}
Bernstein polynomials, long used in splines, computer-aided geometric design, and computer graphics were first introduced more than a century ago~\cite{bernstein1912demo}.  More recently, they have also been considered as a tool for high-order approximation of partial differential equations via the finite element method.  Tensorial decomposition of the simplicial Bernstein polynomials under the Duffy transform~\cite{duffy1982quadrature} has led to fast algorithms with comparable complexity to spectral element methods in rectangular geometry~\cite{ainsworth2011bernstein}.  Moreover, Bernstein properties also lead to special recursive blockwise-structure for finite element matrices~\cite{kirby2017fast,kirby2012fast}.  Bernstein polynomials are not only a tool for standard $C^0$ finite element spaces, but also can be used to represent bases for polynomial differential forms discretizing the de Rham complex~\cite{ainsworth2018bernstein, kirby2014low}.

In this paper, we focus on the inversion of the one-dimensional mass matrix.  In~\cite{kirby2017fast}, we gave recursive, block-structured fast algorithms for mass matrices on the $d$-simplex.  These rely on inversion of the one-dimensional mass matrix as an essential building block, but we did not consider fast algorithms or the underlying stability issues for that case.  We turn to these issues here.  In addition to studying the inverse matrix and fast algorithms for applying it, we also give an argument that the effective conditioning for Bernstein polynomials is less extreme than it might appear.  This amounts to introducing a nonstandard matrix norm in which to consider the problem.  In fact, our main result on conditioning mass matrices -- that in the ``right'' norm the condition number is the square root of the standard matrix 2-norm -- is generic and not limited to the particularities of either univariate polynomials or the Bernstein basis.  We believe this sheds important light on the observation that, despite massive condition numbers, Bernstein polynomials of degrees greater than ten can still give highly accurate finite element methods~\cite{AINSWORTH2019766,kirby2014low}.  It does also, however, suggest an eventual limit to the degree of polynomial absent algorithmic advances.

The mass (or Gram) matrix arises abstractly in the case of finding the best approximation onto a subspace of a generic Hilbert space.
Given a Hilbert space $H$ with inner product $(\cdot, \cdot)$ and a finite-dimensional subspace $V^N$ with basis $\{\phi_i\}_{i=0}^N$, the problem of optimally approximating $u \in H$ by $\sum_{i=0}^N c_i \phi_i$ with respect to the norm of $H$ requires solving the linear system
\begin{equation}
\label{eq:theeq}
M \bfc = \bfb,
\end{equation}
where $M_{ij} = (\phi_i, \phi_j)$ and $\bfb_i = (u, \phi_i)$.

While choosing an orthogonal basis trivializes the inversion of $M$, matrices of this type also arise in other settings that constrain the choice of basis.  For example, finite element methods typically require ``geometrically decomposed'' bases~\cite{arnold2009geometric} that enable enforcement of continuity across mesh elements.  Even when discontinuous bases are admissible for finite element methods, geometric decomposition can be used to simplify boundary operators~\cite{kirby2017fast}.  Moreover, splines and many geometric applications make use of the Bernstein-Bezier basis as local representations.

\section{Notation and Mass matrix}
For integers $n \geq 0$ and $0 \leq i \leq n$, we define the Bernstein polynomial $B^n_i(x)$ as
\begin{equation}
B^n_i(x) = \binom{n}{i} x^i (1-x)^{n-i}.
\end{equation}
Each $B^n_i(x)$ is a polynomial of degree $n$, and the set
$\left\{ B^n_i(x) \right\}_{i=0}^n$ forms a basis for the vector space of polynomials of degree $n$.

If $m\leq n$, then any polynomial expressed in the basis $\left\{B^m_i(x)\right\}_{i=0}^m$ can also be expressed in the basis $\left\{B^n_i(x)\right\}_{i=0}^n$. We denote by $E^{m,n}$ the $(n+1)\times (m+1)$ matrix that maps the coefficients of the degree $m$ representation to the coefficients of the degree $n$ representation. It is remarked in~\cite{farouki2000legendre} that the entries of $E^{m,n}$ are given by
\begin{equation}
\label{elevation}
E^{m,n}_{ij} = \frac{\binom{m}{j}\binom{n-m}{i-j}}{\binom{n}{i}},
\end{equation}
with the standard convention that $\binom{n}{i} = 0$ for $i < 0$ or $i > n$.

We will also need reference to the Legendre polynomials, mapped from their typical home on $[-1, 1]$ to $[0, 1]$.  We let $L^n(x)$ denote the Legendre polynomial of degree $n$ over $[0, 1]$, scaled so that $L^n(1)=1$ and
\begin{equation}
\label{legendreL2norm}
\left\| L^n\right\|^2_{L^2} = \frac{1}{2n+1}.
\end{equation}
It was observed in~\cite{farouki2000legendre} that $L^n(x)$ has the following representation:
\begin{equation}
\label{legendrebernstein}
L^n(x) = \sum_{i=0}^n (-1)^{n+i} \binom{n}{i} B^n_i(x).
\end{equation}

The Bernstein mass matrix $M^n$ is given by
\begin{equation}
M^n_{ij} = \int_0^1 B^n_i(x) B^n_j(x) dx,
\end{equation}
which can be exactly computed~\cite{kirby2011fast} as
\begin{equation}
M^n_{ij} = \binom{n}{i} \binom{n}{j} \frac{(2n-i-j)!(i+j)!}{(2n+1)!}.
\end{equation}

The mass matrix, whether with Bernstein or any other polynomials, plays an important role in connecting the $L^2$ topology on the finite-dimensional space to linear algebra.  To see this, we first define mappings connecting polynomials of degree $n$ to $\mathbb{R}^{n+1}$.  Given any $\bfc \in \mathbb{R}^{n+1}$, we let $\pi(\bfc)$ be the polynomial expressed in the Bernstein basis with coefficients contained in $\bfc$:
\begin{equation}
\pi(\bfc)(x) = \sum_{i=0}^n \bfc_i B^n_i(x).
\end{equation}
We let $\bfpi$ be the inverse of this mapping, sending any polynomial of degree $n$ to the vector of $n+1$ coefficients with respect to the Bernstein basis.

Now, let $p(x)$ and $q(x)$ be polynomials of degree $n$ with expansion coefficients $\bfpi(p) = \bfp$ and $\bfpi(q) = \bfq$.  Then the $L^2$ inner product of $p$ and $q$ is given by the $M^n$-weighted inner product of $\bfp$ and $\bfq$, for 
\begin{equation}
\label{eq:innprod}
\int_0^1 p(x) q(x) dx
=
\sum_{i,j=0}^n \bfp_i \bfq_j \int_0^1 B^n_i(x) B^n_j(x) dx
= \bfp^T M^n \bfq.
\end{equation}
Similarly, if
\begin{equation}
\| \bfp \|_{M^n} = \sqrt{ \bfp^T M^n \bfp }
\end{equation}
is the $M^n$-weighted vector norm, then we have for $p = \pi(\bfp)$,
\begin{equation}
\| p \|_{L^2} = \| \bfp \|_{M^n}.
\end{equation}

It was shown in~\cite{kirby2011fast} that if $m\leq n$, then
\begin{equation}
\label{mass_elev}
M^m = \left(E^{m,n}\right)^T M^n E^{m,n}.
\end{equation}

Further, in~\cite{kirby2011fast} it was observed that the mass matrix is, up to a row and column scaling, a Hankel matrix (constant along antidiagonals).  Let $\Delta^n$ be the $(n+1)\times(n+1)$ diagonal matrix with binomial coefficients on the diagonal:
\begin{equation}
\Delta^n_{ij}
= \begin{cases}
\binom{n}{i}, & i = j; \\
0, & \text{otherwise}.
\end{cases}
\end{equation}
Then, let $\widetilde{M}^n$ be defined by
\begin{equation}
\widetilde{M}^n_{ij} = \frac{(2n-i-j)!(i+j)!}{(2n+1)!},
\end{equation}
so that
\begin{equation}
\label{eq:factorm}
M^n = \Delta^n \widetilde{M}^n \Delta^n.
\end{equation}
Since  $\widetilde{M}^n$ depends only on the sum $i+j$ and not $i$ and $j$ separately, it is a Hankel matrix.  Therefore, $\widetilde{M}^n$ and hence $M^n$ can be applied to a vector in $\mathcal{O}(n \log n)$ operations via a circulant embedding and fast Fourier transforms~\cite{luk2000fast}, although the actual point where this algorithm wins over the standard one may be at a rather high degree.

In~\cite{kirby2017fast}, we gave a spectral characterization of the mass matrix on simplices of any dimension.  For the one-dimensional case, the eigenvalues and eigenvectors are:
\begin{theorem}
\label{thm:masseig}
The eigenvalues of $M^n$ are $\{\lambda_i^n\}_{i=0}^n$, where
\begin{equation}
\lambda^n_i = \frac{(n!)^2}{(n+i+1)!(n-i)!},
\end{equation}
and the eigenvector of $M^n$ corresponding to $\lambda_i^n$ is $E^{i,n}\bfpi(L^i)$.
\end{theorem}

\section{Characterizing the inverse matrix}
In this section, we present several approaches to constructing and applying the inverse of the mass matrix.

\subsection{A formula for the inverse}
The following result (actually, a generalization) is derived in~\cite{lewanowicz2011bezier}:\begin{theorem}
\label{thm:inverseformula}
\begin{equation}
%\label{thm:inverseformula}
\left(M^n\right)^{-1}_{ij} = \frac{(-1)^{i+j}}{\binom{n}{i}\binom{n}{j}}\sum_{k=0}^n (2k+1-i+j)\binom{n+1}{i-k}^2\binom{n+1}{j+k+1}^2.
\end{equation}
\end{theorem}
This result is obtained by characterizing (possibly constrained) Bernstein dual polynomials via Hahn orthogonal polynomials.  Much earlier, Farouki~\cite{farouki2000legendre} gave a characterization of the standard dual polynomials by more direct means.  Lu~\cite{lu2015gram} builds on the work in~\cite{lewanowicz2011bezier} to give a recursive formula for applying this inverse.

In this section, we summarize a proof of Theorem~\ref{thm:inverseformula} based on Farouki's representation of the dual basis in Subsection~\ref{ssec:inversedual}, and then give an alternative derivation of the result based on B\'ezoutians in Subsection~\ref{ssec:bezout}.  Since the mass matrix is diagonal in the Legendre basis, we can also characterize the inverse operator by converting from Bernstein to Legendre representation, multiplying by a diagonal, and then converting back.  This approach, equivalent to the spectral decomposition, is given in Subsection~\ref{ssec:convert}.

\subsubsection{Inversion via the dual basis}
\label{ssec:inversedual}

The dual basis to $\left\{ B^n_i(x)\right\}_{i=0}^n$ is the set of polynomials $\left\{ d^n_i(x)\right\}_{i=0}^n$ satisfying
\begin{equation}
\label{dual}
\int_0^1 B^n_i(x)d^n_j(x) dx =
\begin{cases}
1, & i=j; \\
0, & \text{otherwise}.
\end{cases}
\end{equation}
If we consider $\bfd^{n,j}=\bfpi(d^n_j)$, then
\begin{equation}
\int_0^1 B^n_i(x)d^n_j(x) dx = \sum_{k=0}^n \bfd^{n,j}_k \int_0^1 B^n_i(x)B^n_k(x) dx = \left(M^n\bfd^{n,j}\right)_i,
\end{equation}
and hence (\ref{dual}) implies that $\left( M^n\right)^{-1}_{ij} = \bfd^{n,j}_i$. Following Farouki's representation of the dual basis coefficients~\cite{farouki2000legendre}, this gives us that
\begin{equation}
\left( M^n\right)^{-1}_{ij} = \frac{(-1)^{i+j}}{\binom{n}{i}\binom{n}{j}} \sum_{k=0}^n (2k+1)\binom{n+k+1}{n-j}\binom{n-k}{n-j}\binom{n+k+1}{n-i}\binom{n-k}{n-i}.
\end{equation}
Applying a few binomial identities gives the representation in Theorem~\ref{thm:inverseformula}.

\subsubsection{Inversion via Bezoutians}
\label{ssec:bezout}

In this section, we detail an alternate derivation of Theorem~\ref{thm:inverseformula} via B\'ezout matrices, which we will use in Section~\ref{ssec:factor} to decompose the inverse in terms of diagonal, Hankel, and Toeplitz matrices.

If $u(t)=\sum_{i=0}^{n+1} u_i t^i$ and $v(t)=\sum_{i=0}^{n+1} v_i t^i$ are polynomials of degree less than or equal to $(n+1)$ expressed in the monomial basis, then the B\'ezout matrix generated by $u$ and $v$, denoted $\Bez(u,v)$, is the $(n+1)\times (n+1)$ matrix with entries $b_{ij}$ satisfying
\begin{equation}
\frac{u(s)v(t)-u(t)v(s)}{s-t}=\sum_{i,j=0}^n b_{ij}s^i t^j.
\end{equation}
By comparing coefficients, we have that a closed form expression for the entries is given by
\begin{equation}
\label{bezout_entries}
b_{ij} = \sum_{k=0}^{\min\left\{i,n-j\right\}}\left( u_{j+k+1}v_{i-k}-u_{i-k}v_{j+k+1}\right).
\end{equation}

Heinig and Rost~\cite{heinig1984algebraic} gave a formula for the inverse of a Hankel matrix in terms of a B\'ezout matrix.
\begin{theorem}
\label{heinigrost}
If $H$ is an $(n+1)\times (n+1)$ nonsingular Hankel matrix, and $\widehat{H}$ is any $(n+2)\times (n+2)$ nonsingular Hankel extension of $H$ obtained by appending a row and column to $H$, then
\begin{equation}
\label{eq:heinigrost}
H^{-1} = \frac{1}{v_{n+1}}\Bez(u,v),
\end{equation}
where $u$ is the polynomial whose coefficients are given by the last column of $H^{-1}$, and $v$ is the polynomial whose coefficients are given by the last column of $\widehat{H}^{-1}$.
\end{theorem}
Since the matrices $H$ and $\widehat{H}$ are of different sizes, the corresponding polynomials have different degrees. We reconcile this difference by elevating by one degree in the monomial basis (that is, appending a zero to the vector of coefficients).

The next few results are dedicated to finding the $u^n$ and $v^{n+1}$ that correspond to $\widetilde{M}^n$. We begin by showing that the null space of $\left(E^{n-1,n}\right)^T$ is spanned by the eigenvector corresponding to $\lambda^n_n$. We then use this to project our candidate $\bfy^n$ for the last column of $\left(M^n\right)^{-1}$ onto the null space and its orthogonal complement. This decomposition gives a recurrence relation for the $\bfy^n$, which will be the foundation for an inductive proof that $\bfy^n$ is the last column of $\left(M^n\right)^{-1}$. The coefficients of $u^n$ are then obtained via~\eqref{eq:factorm}.

\begin{lemma}
\label{lem:null}
The null space of $\left(E^{n-1,n}\right)^T$ is spanned by $\bfpi(L^n)$.
\end{lemma}

\begin{proof}
By~\eqref{elevation}, we have that $\left(E^{n-1,n}\right)^T$ is upper bidiagonal with nonzero entries on the main diagonal and the superdiagonal, and so the null space of $\left(E^{n-1,n}\right)^T$ is one-dimensional. Therefore, it is enough to show that $\bfpi(L^n)$ belongs to the null space of $\left(E^{n-1,n}\right)^T$.

Let $\bfq=E^{n-1,n}\bfp$ be in the range of $E^{n-1,n}$. This means that $p=\pi(\bfp)$ is a polynomial of degree at most $n-1$, and hence
\begin{equation}
0 = \int_0^1 p(x)L^n(x) dx.
\end{equation}
Therefore, Theorem~\ref{thm:masseig} and~\eqref{eq:innprod} imply that
\begin{equation}
0 = (E^{n-1,n}\bfp)^TM^n\bfpi(L^n) = \lambda^n_n \bfq^T\bfpi(L^n)
= \frac{(n!)^2}{(2n+1)!} \bfq^T \bfpi(L^n).
\end{equation}
This implies that $\bfpi(L^n)$ is orthogonal in the Euclidean inner product to the range of $E^{n-1,n}$, and so $\bfpi(L^n)$ belongs to the null space of $\left(E^{n-1,n}\right)^T$.
\end{proof}

\begin{lemma}
\label{lem:recurrence}
For $n \geq 0$, 
let $\bfy^n$ be given by
\begin{equation}
\bfy^n_i = (-1)^{n+i}(n+1)\binom{n+1}{i}.
\end{equation}
Then, for $n \geq 1$, 
\begin{equation}
\label{eq:recurrence}
\bfy^n = (2n+1)\bfpi(L^n)+E^{n-1,n}\bfy^{n-1}.
\end{equation}
\end{lemma}

\begin{proof}

By~\eqref{elevation} and~\eqref{legendrebernstein}, we have that for each $0\leq i\leq n$,
\begin{align*}
\left[(2n+1)\bfpi(L^n)+E^{n-1,n}\bfy^{n-1}\right]_i &= (2n+1)\bfpi(L^n)_i+\frac{i}{n}\bfy^{n-1}_{i-1}+\frac{n-i}{n}\bfy^{n-1}_i \\
&= (2n+1)(-1)^{n+i}\binom{n}{i} + i(-1)^{n+i}\binom{n}{i-1}-(n-i)(-1)^{n+i}\binom{n}{i} \\
&= (-1)^{n+i}\left[ (n+i+1) \binom{n}{i} + i \binom{n}{i-1}\right] \\
&= (-1)^{n+i} (n+1) \binom{n+1}{i} \\
&= \bfy^n_i.
\end{align*}

\end{proof}

\begin{proposition}
\label{prop:lastcol}
Let $\bfy^n$ be as in Lemma~\ref{lem:recurrence}. Then 
\begin{equation}
M^n\bfy^n=\bfe^n,
\end{equation}
where
\begin{equation}
\bfe^n_i
= \begin{cases}
1, & i=n; \\
0, & \text{otherwise}.
\end{cases}
\end{equation}
\end{proposition}

\begin{proof}

We show the result by induction on $n$. Clearly, the result is true for $n=0$. So suppose $M^{n-1}\bfy^{n-1}=\bfe^{n-1}$ for some $n\geq 1$. We show that $M^n\bfy^n=\bfe^n$ by showing that $M^n\bfy^n-\bfe^n$ belongs to both the null space of $\left(E^{n-1,n}\right)^T$ and its orthogonal complement with respect to the Euclidean inner product.

Multiplying~\eqref{eq:recurrence} on the left by $M^n$ gives us that
\begin{equation}
\label{stepping_stone}
M^n\bfy^n = (2n+1)\lambda^n_n\bfpi(L^n)+M^nE^{n-1,n}\bfy^{n-1}.
\end{equation}
Since the null space of $\left(E^{n-1,n}\right)^T$ is spanned by $\bfpi(L^n)$, multiplying the previous equation on the left by $\left(E^{n-1,n}\right)^T$ and using~\eqref{mass_elev} and the induction hypothesis gives us that
\begin{equation}
\left(E^{n-1,n}\right)^TM^n\bfy^n = \bfe^{n-1}.
\end{equation}
Since $\bfe^{n-1}=\left( E^{n-1,n}\right)^T\bfe^n$, the previous equation implies that $M^n\bfy^n-\bfe^n$ belongs to the null space of $\left(E^{n-1,n}\right)^T$.

On the other hand, (\ref{stepping_stone}) implies that
\begin{align*}
\bfpi(L^n)^T(M^n\bfy^n-\bfe^n) &= \bfpi(L^n)^T((2n+1)\lambda^n_n\bfpi(L^n)+M^nE^{n-1,n}\bfy^{n-1}-\bfe^n) \\
&= (2n+1)\lambda^n_n \bfpi(L^n)^T\bfpi(L^n) + \bfpi(L^n)^TM^nE^{n-1,n}\bfy^{n-1} - \bfpi(L^n)_n.
\end{align*}
By~\eqref{legendrebernstein}, we have that
\begin{equation}
\bfpi(L^n)^T\bfpi(L^n) = \sum_{i=0}^n \binom{n}{i}^2 = \binom{2n}{n} = \frac{1}{(2n+1)\lambda^n_n}
\end{equation}
and $\bfpi(L^n)_n = 1$. Therefore,
\begin{equation}
\bfpi(L^n)^T(M^n\bfy^n-\bfe^n) = \bfpi(L^n)^TM^nE^{n-1,n}\bfy^{n-1}.
\end{equation}
Since $\bfpi(L^n)$ is orthogonal in the Euclidean inner product to the range of $E^{n-1,n}$,
\begin{equation}
\bfpi(L^n)^TM^nE^{n-1,n}\bfy^{n-1} = (M^n\bfpi(L^n))^TE^{n-1,n}\bfy^{n-1} = \lambda^n_n \bfpi(L^n)^TE^{n-1,n}\bfy^{n-1} = 0,
\end{equation}
and hence $\bfpi(L^n)^T(M^n\bfy^n-\bfe^n)=0$. Therefore, $M^n\bfy^n-\bfe^n$ also belongs to the orthogonal complement of the null space of $\left( E^{n-1,n}\right)^T$.

\end{proof}

Since $\Delta^n_{nn} = 1$, $M^n\bfy^n=\bfe^n$ if and only if $\widetilde{M}^n\Delta^n\bfy^n=\bfe^n$.  Therefore, the coefficients of $u^n$ are given by 
\begin{equation}
u^n_i = \left(\Delta^n\bfy^n\right)_i=(-1)^{n+i}(n+1)\binom{n}{i}\binom{n+1}{i}.
\end{equation}
We now find the coefficients of $v^{n+1}$. By Theorem~\ref{heinigrost}, the coefficients are given by the last column of the inverse of any nonsingular Hankel extension of $\widetilde{M}^n$.  We use this freedom to choose an extension such that the coefficients of $v^{n+1}$ are given by elevating $u^n$ by one degree in the monomial basis and then reversing the order of the entries. To describe this process, we introduce the following notation: given a vector $\bfx$ of length $(n+1)$, denote by $\bfx^P$ the vector of length $(n+2)$ given by
\begin{equation}
\bfx^P = P\begin{pmatrix} \bfx \\ 0 \end{pmatrix},
\end{equation}
where
\begin{equation}
P 
= \begin{pmatrix} 
& & & & 1 \\
& & & 1 & \\
& & \iddots & & \\
& 1 & & & \\
1 & & & & 
\end{pmatrix}
\end{equation}
is the $(n+2)\times (n+2)$ exchange matrix.

\begin{lemma}
\label{lem:extension}
If $H$ is an $(n+1)\times (n+1)$ nonsingular Hankel matrix of the form
\begin{equation}
H
= \begin{pmatrix}
h_0 & h_1 & h_2 & \cdots & h_n \\
h_1 & h_2 & h_3 & \cdots & h_{n-1} \\
h_2 & h_3 & h_4 & \cdots & h_{n-2} \\
\vdots & \vdots & \vdots & \iddots & \vdots \\
h_n & h_{n-1} & h_{n-2} & \cdots & h_0
\end{pmatrix},
\end{equation} and $\bfx$ satisfies
\begin{equation}
H\bfx = \bfe^n
\end{equation}
with $\bfx_0\neq 0$, then there exists an $(n+2)\times (n+2)$ Hankel extension $\widehat{H}$ of $H$ such that
\begin{equation}
\label{eq:extension}
\widehat{H}\bfx^P=\bfe^{n+1}.
\end{equation}
\end{lemma}

\begin{proof}

Since $\bfx^P$ is the action of $\begin{pmatrix} \bfx & 0\end{pmatrix}^T$ under the exchange matrix $P$, we can instead apply the exchange matrix to $\widehat{H}$ and show that there exist $\alpha$ and $\beta$ such that 
\begin{equation}
\begin{pmatrix}
h_{n-1} & h_n & \cdots & h_3 & h_2 & h_1 & h_0 \\
h_{n-2} & h_{n-1} & \cdots & h_4 & h_3 & h_2 & h_1 \\
h_{n-3} & h_{n-2} & \cdots & h_5 & h_4 & h_3 & h_2 \\
h_{n-4} & h_{n-3} & \cdots & h_6 & h_5 & h_4 & h_3 \\
\vdots & \vdots & \ddots & \vdots & \vdots & \vdots & \vdots \\
\alpha & h_0 & \cdots & h_{n-3} & h_{n-2} & h_{n-1} & h_n \\
\beta & \alpha & \cdots & h_{n-4} & h_{n-3} & h_{n-2} & h_{n-1}
\end{pmatrix}
\begin{pmatrix}
\bfx_0 \\
\bfx_1 \\
\bfx_2 \\
\bfx_3 \\
\vdots \\
\bfx_n \\
0
\end{pmatrix}
= \begin{pmatrix}
0 \\
0 \\
0 \\
0 \\
\vdots \\
0 \\
1
\end{pmatrix}.
\end{equation}
The first $n$ equations are satisfied by assumption. Therefore, choosing
\begin{equation}
\alpha = -\frac{1}{\bfx_0}\sum_{i=0}^{n-1} h_i\bfx_{i+1}
\end{equation}
and
\begin{equation}
\beta = \frac{1}{\bfx_0}\left( 1 - \alpha\bfx_1 - \sum_{i=0}^{n-2} h_i\bfx_{i+2}\right)
\end{equation}
gives the desired result.

\end{proof}

Therefore, the coefficients of $v^{n+1}$ are given by
\begin{equation}
v^{n+1}_i = \left( \Delta^n\bfy^n\right)^P_i = (-1)^{i+1}(n+1)\binom{n+1}{i}\binom{n}{i-1},
\end{equation}
and hence Theorem~\ref{heinigrost} and (\ref{bezout_entries}) imply that
\begin{theorem}
\label{thm:hankelinverse}
\begin{align*}
\left(\widetilde{M}^n\right)^{-1}_{ij} &= \frac{1}{v^{n+1}_{n+1}}\sum_{k=0}^{\min\left\{i,n-j\right\}}\left( u^n_{j+k+1}v^{n+1}_{i-k}-u^n_{i-k}v^{n+1}_{j+k+1}\right) \\ 
&= (-1)^{i+j} \sum_k \left( 2k+1-i+j\right) \binom{n+1}{i-k}^2\binom{n+1}{j+k+1}^2.
\end{align*}
\end{theorem}
Applying the identity $\left(M^n\right)^{-1} = \left( \Delta^n\right)^{-1} \left( \widetilde{M}^n\right)^{-1} \left( \Delta^n\right)^{-1}$ gives the result from Theorem~\ref{thm:inverseformula}.

\subsubsection{Bernstein-Legendre conversion}
\label{ssec:convert}

We can use Theorem~\ref{thm:masseig} to construct the spectral (or eigenvector) decomposition 
\begin{equation}
M^n = Q^n \Lambda^n (Q^n)^T,
\end{equation}
where $\Lambda^n$ contains the eigenvalues and $Q^n$ is an orthogonal matrix of eigenvectors. Since $M^n$ is symmetric and all of its eigenvalues are distinct, the eigenvectors are orthogonal. Therefore, all that remains is to normalize the eigenvectors in the 2-norm.
\begin{proposition}
\label{prop:legendre2norm}
\begin{equation}
\label{eq:legendre2norm}
\left\| E^{k,n}\bfpi(L^k)\right\|^2_2 = \frac{1}{(2k+1)\lambda^n_k}.
\end{equation}
\end{proposition}

\begin{proof}

The result follows from~\eqref{legendreL2norm} and~\eqref{eq:innprod}:
\begin{equation}
\frac{1}{2k+1} = \| L^k \|^2_{L^2} = 
\| E^{k,n} \bfpi(L^k) \|_{M^n}^2
= \left( E^{k,n}\bfpi(L^k) \right)^T M^n E^{k,n}\bfpi(L^k)
= \lambda^n_k \left\| E^{k,n}\bfpi(L^k)\right\|^2_2.
\end{equation}
\end{proof}

Proposition~\ref{prop:legendre2norm} implies that
\begin{equation}
\label{eq:Qn}
Q^n = 
\begin{pmatrix}
\sqrt{\lambda^n_0}E^{0,n}\bfpi(L^0) & | & \sqrt{3\lambda^n_1}E^{1,n}\bfpi(L^1) & | & \cdots & | & \sqrt{(2n+1)\lambda^n_n}\bfpi(L^n)
\end{pmatrix}
\end{equation}
and
\begin{equation}
\Lambda^n = \diag\left(\lambda^n_0, \lambda^n_1, \dots, \lambda^n_n\right).
\end{equation}
%It follows that $(M^n)^{-1} = Q^n\left( \Lambda^n\right)^{-1}\left( Q^n\right)^T$.

Equation~\eqref{eq:Qn} characterizes $Q^n$ and suggests an algorithm for its construction -- begin with $\bfpi(L^i)$ and elevate it to degree $n$.  This requires $\mathcal{O}(n^3)$ operations -- there are $\mathcal{O}(n)$ columns and each one needs to be elevated $\mathcal{O}(n)$ times at $\mathcal{O}(n)$ operations per elevation.  However, we can also adapt the classical 3-term recurrence for the Legendre polynomials to give a $\mathcal{O}(n^2)$ process for constructing $Q^n$.

We begin the process by constructing the degree $n$ representation of $L^0$ -- which is simply the vector of ones.  We can similarly construct the coefficients for $L^1$.  Then, we proceed inductively.  Assuming that $L^i$ has been constructed (for $1 \leq i < n$),  the critical step in the recurrence is to multiply $L^i(x)$ by $x$.  Given any polynomial of degree $n$ 
\[
p(x) = \sum_{i=0}^n p_i B^n_i(x),
\]
we have
\[
x p(x) = \sum_{i=0}^n p_i \binom{n}{i} x^{i+1} (1-x)^{n-i}
= \sum_{i=0}^{n+1} \tilde{p}_i B^{n+1}_i(x),
\]
where $\tilde{p}_0 = 0$ and $\tilde{p}_i = \tfrac{i p_{i-1}}{n+1}$ for $1 \leq i \leq n+1$.  This $\mathcal{O}(n)$ computes $x p(x)$ in the degree $n+1$ basis, but we need $x L^i(x)$ in the Bernstein basis of degree $n$.  That is, if $\tilde{\mathbf{p}}$ holds the degree $n+1$ Bernstein coefficients, we need to find $\mathbf{q}$ such that
\begin{equation}
E^{n,n+1} \mathbf{q} = \tilde{\mathbf{p}},
\end{equation}
which is a consistent system of $n+2$ equations in $n+1$ unknowns.  Several $\mathcal{O}(n)$ algorithms for this are possible, but Gaussian elimination on the (tridiagonal) normal equations seems stable in practice.  Algorithm~\ref{alg:Qn} summarizes building the entire matrix $Q^n$ using the three-term recurrence and the tridiagonal degree reduction, where $E$ denotes the matrix $E^{n,n+1}$.

\begin{algorithm}
\caption{Builds the orthogonal matrix $Q^n$ of eigenvectors of $M^n$}
\label{alg:Qn}
\begin{algorithmic}
\STATE $Q^n[:,0]\gets E^{0,n}\bfpi(L^0)$
\STATE $Q^n[:,1]\gets E^{1,n}\bfpi(L^1)$
\FOR{$j=2$ \TO $n-1$}
\STATE{$\widetilde{\mathbf{p}}\gets \bfpi(x\pi(Q[:,j-1])(x))$}
\STATE{$\mathbf{q}\gets \left(E^TE\right)^{-1}E^T\widetilde{\mathbf{p}}$}
\STATE{$Q^n[:,j]\gets \frac{2j-1}{j}\left(2\mathbf{q}-Q^n[:,j-1]\right)-\frac{j-1}{j}Q^n[:,j-2]$}
\ENDFOR
\STATE $Q^n[:,n]\gets \bfpi(L^n(x))$
\FOR{$j=0$ \TO $n$}
\STATE{$Q^n[:,j]\gets \sqrt{(2j+1)\lambda^n_j} Q^n[:,j]$}
\ENDFOR
\end{algorithmic}
\end{algorithm}
Consequently, $(M^n)^{-1} = Q^n \left( \Lambda^n \right)^{-1} \left( Q^n \right)^T$ can be applied to a vector by multiplying by two dense orthogonal matrices and a diagonal scaling.

\subsection{Decomposing the inverse}
\label{ssec:factor}
The proof of Theorem~\ref{thm:hankelinverse} suggests the following decomposition of $(M^n)^{-1}$ into diagonal, Toeplitz, and Hankel matrices:
\begin{theorem}
\label{thm:factorinverse}
Let $T^n$ and $\widetilde{T}^n$ be the Toeplitz matrices given by
\begin{equation}
T^n_{ij} = (-1)^{i-j} \binom{n+1}{i-j}^2\qquad \text{and}\qquad \widetilde{T}^n_{ij} = (i-j)T^n_{ij},
\end{equation}
and let $H^n$ and $\widetilde{H}^n$ be the Hankel matrices given by
\begin{equation}
H^n_{ij} = (-1)^{i+j+1}\binom{n+1}{i+j+1}^2 \qquad \text{and} \qquad \widetilde{H}^n_{ij} = (i+j+1)H^n_{ij}.
\end{equation}
Then
\begin{equation}
\label{eq:factorinverse}
\left(M^n\right)^{-1} = \left( \Delta^n\right)^{-1}\left[ \widetilde{T}^nH^n-T^n\widetilde{H}^n\right] \left( \Delta^n \right)^{-1}.
\end{equation}
\end{theorem}

\begin{proof}
By Theorem~\ref{thm:hankelinverse}, we have that
\begin{align*}
\left(\widetilde{M}^n\right)^{-1}_{ij} &= (-1)^{i+j} \sum_k \left( 2k+1-i+j\right) \binom{n+1}{i-k}^2\binom{n+1}{j+k+1}^2 \\
&= \sum_k \left[ (-1)^{i-k}(i-k)\binom{n+1}{i-k}^2\right] \left[ (-1)^{j+k+1}\binom{n+1}{j+k+1}^2\right] \\ 
&\qquad - \sum_k \left[ (-1)^{i-k}\binom{n+1}{i-k}^2\right] \left[ (-1)^{j+k+1}(j+k+1)\binom{n+1}{j+k+1}^2\right].
\end{align*}
Therefore, $\left(\widetilde{M}^n\right)^{-1} = \widetilde{T}^nH^n-T^n\widetilde{H}^n$, and so the result follows from~\eqref{eq:factorm}. 
\end{proof}
This result implies a superfast $\mathcal{O}(n \log n)$ algorithm, but our numerical experiments suggest that it is highly unstable.

\section{Applying the inverse}
\label{sec:apply}
Now, we describe several approaches to applying $M^{-1}$ to a vector. 

\paragraph{Cholesky factorization}
$M^n$ is symmetric and positive definite and therefore admits a Cholesky factorization
\begin{equation}
M^n = L L^T,
\end{equation}
where $L$ is lower triangular with positive diagonal entries~\cite{strang}.  Widely available in libraries, computing $L$ requires $\mathcal{O}(n^3)$ operations, and each of the subsequent triangular solves require $\mathcal{O}(n^2)$ operations to perform.  Our numerical results below suggest that it is one of the more stable and accurate methods under consideration, although our technique based on the eigendecomposition has similar accuracy and $\mathcal{O}(n^2)$ complexity for the startup phase.

\paragraph{Exact inverse}
In light of Theorem~\ref{thm:inverseformula}, we can directly form $M^{-1}$.  Since the formula for each entry requires a sum, forming the inverse requires $\mathcal{O}(n^3)$ operations.  The inverse matrix can then be applied to any vector using the standard algorithm and in $\mathcal{O}(n^2)$ operations.  This has the same startup and per-solve complexity as the Cholesky factorization, although the constants are different.

\paragraph{Spectral decomposition}
In Section~\ref{ssec:convert}, we showed how to compute the eigendecomposition of $M^n$ and hence express its inverse as
\begin{equation}
\left( M^n \right) ^{-1} = Q^n  \left( \Lambda^n \right)^{-1} \left(Q^n \right)^T.
\end{equation}
The inverse can be applied by two (dense) matrix multiplications and a diagonal scaling, requiring $\mathcal{O}(n^2)$ operations.  Thanks to Algorithm~\ref{alg:Qn}, we have only an $\mathcal{O}(n^2)$ startup phase.

\paragraph{DFT-based application}
Theorem~\ref{thm:factorinverse} implies that we can multiply by the inverse of $M^n$ in $\mathcal{O}(n \log n)$ operations.  We invert $\Delta^n$ onto a given vector, and then all of the Toeplitz and Hankel matrices can be applied via circulant embedding before inverting $\Delta^n$ onto the result.  Moreover, the operations can be fused together so that some FFTs are re-used and FFT/inverse FFT pairs cancel in~\eqref{eq:factorinverse}.  However, our numerical results reveal this approach to be quite unstable in practice, becoming wildly inaccurate long before one could hope to win from the super-fast algorithm.  Therefore, we do not go into much detail on this approach.

\section{Conditioning and accuracy}
\label{sec:cond}
As a corollary of Theorem~\ref{thm:masseig}, $M^n$ is terribly ill-conditioned as $n$ increases.  For each $n$, the minimal eigenvalue occurs when $i=n$, which is
\begin{equation}
\lambda_{\min}(M^n) = \frac{(n!)^2}{(2n+1)!}.
\end{equation}
Similarly, the maximal eigenvalue occurs when $i=0$:
\begin{equation}
\lambda_{\max}(M^n) = \frac{(n!)^2}{(n+1)!n!} = \frac{1}{n+1}.
\end{equation}

Given these extremal eigenvalues, the 2-norm condition number is
\begin{equation}
\label{eq:kappa2}
\kappa_2(M^n) = \frac{(2n+1)!}{(n+1)! n!}.
\end{equation}

While this conditioning seems spectacularly bad, in practice, the Bernstein basis seems to give entirely satisfactory results at moderately high orders of approximation~\cite{kirby2014low}.  Here, we hope to give at least a partial explanation of this phenomenon.   Recall that solving $M^n\bfc = \bfb$ exactly yields the Bernstein expansion coefficients of the best $L^2$ polynomial approximation to a function $f$ whose moments against the Bernstein basis are contained in $\bfb$.  Consequently, if our solution process yields some $\bfc + \delta \bfc$ instead of $\bfc$, the $L^2$ norm of the polynomial encoded by $\delta \bfc$ has more direct relevance than the size of $\bfc$ in the Euclidean norm.  On the other hand, the perturbation $\delta \bfb$ exists as an array of numbers (say, the roundoff error in computing the moments of $f$ by numerical integration), and we continue to use the Euclidean norm here.

In standard floating point analysis, suppose we have computed the solution to a perturbed solution
\begin{equation}
%\label{eq:perturbed}
M^n \left(\bfc + \delta \bfc\right) = \bfb + \delta \bfb,
\end{equation}
and while a backward stable algorithm yields a small $\delta \bfb$, the perturbation in the solution $\delta \bfc$ might itself be significant.  Equivalently, the perturbation $\delta\bfc$ satisfies
\begin{equation}
\label{eq:perturbed}
M^n \delta\bfc = \delta\bfb.
\end{equation}

While classical error perturbation and conditioning analysis would estimate $\| \delta \bfc \|$ in the Euclidean norm, we wish to consider $ \| \delta \bfc \|_{M^n}$, which is the $L^2$ norm of the perturbation in the computed polynomial.

Taking the inner product of~\eqref{eq:perturbed} with $\delta\bfc$, we have:
\begin{equation}
\| \delta \bfc \|^2_{M^n} =
\delta\bfc^T \delta\bfb
\leq \| \delta \bfc \|_2 \| \delta \bfb \|_2
\leq \| (M^n)^{-1} \|_2 \| \delta \bfb \|_2^2,
\end{equation}
so that
\begin{equation}
\label{eq:Minvboundgood}
\| \delta \bfc \|_{M^n} \leq \left(\sqrt{\| (M^n)^{-1} \|_2} \right) \| \delta \bfb \|_2 = \sqrt{\lambda_{\min}(M^n)} \| \delta \bfb \|_2.
\end{equation}
In other words, for a perturbation $\delta \bfb$ of fixed 2-norm, the perturbation $\delta \bfc$ is much smaller measured in the $M^n$ norm than in the 2-norm -- the amplification factor is only the square root as large.

This discussion suggests the following matrix norm.  We think of $M^n$ as an operator mapping $\mathbb{R}^{n+1}$ onto itself.  However, we equip the domain with the $M^n$-weighted norm and the range with the Euclidean norm.  Then, we define the operator norm
\begin{equation}
\label{eq:Mto2}
\| A \|_{M^n\rightarrow 2}
= \max_{\|\bfc\|_{M^n} = 1} \| A \bfc \|_2,
\end{equation}
and going the opposite direction,
\begin{equation}
\label{eq:2toM}
\| A \|_{2\rightarrow M^n}
= \max_{\|\bfc\|_2 = 1} \| A \bfc \|_{M^n}.
\end{equation}
We will also need the matrix norm
\begin{equation}
\| A \|_{M^n} = \max_{\| \bfc \|_{M^n} = 1} \| A \bfc \|_{M^n},
\end{equation}
using the $M^n$-weight in both the domain and range.

In light of~\eqref{eq:2toM}, we can interpret~\eqref{eq:Minvboundgood} as
\begin{proposition}
The norm of $(M^{n})^{-1}$ satisfies
\begin{equation}
\| (M^n)^{-1} \|_{2\rightarrow M^n} = \sqrt{\lambda_{\min}(M^n)}.
\end{equation}
\end{proposition}
We can, by the spectral decomposition, give a similar bound for $M^n$ in the $M\rightarrow 2$ norm:
\begin{proposition}
The norm of $M^n$ satisfies
\begin{equation}
\| (M^n) \|_{M^n\rightarrow 2} = \sqrt{\lambda_{\max}(M^n)}.
\end{equation}
\end{proposition}
\begin{proof}
Since $M^n$ is symmetric and positive-definite, it has a well-defined positive square root via the spectral decomposition.  For $\bfc \in \mathbb{R}^{n+1}$, we write
\begin{equation}
\begin{split}
\| M^n \bfc \|_{2} &= \left( M^n \bfc \right)^T \left( M^n \bfc \right)
= \bfc^T M^2 \bfc \\
&= \left( \sqrt{M^n} \bfc \right)^T M^n \left( \sqrt{M^n} \bfc \right)
= \| \sqrt{M^n} \bfc \|_{M^n} \\
&\leq \| \sqrt{M^n} \|_{M^n} \| \bfc \|_{M^n}.
\end{split}
\end{equation}
Now, we can characterize the weighted norm of the square root matrix via the spectral decomposition $M^n = Q^n \Lambda^n (Q^n)^T$.  

Note that for any $\bfc \in \mathbb{R}^{n+1}$, we have its $M^n$ norm as
\begin{equation}
\| \bfc \|_{M^n}^2
= \bfc^T M^n \bfc
= \left( (Q^n)^T \bfc \right) \Lambda \left( (Q^n)^T \bfc \right).
\end{equation}
Letting $\bfd = (Q^n)^T \bfc$,
\begin{equation}
\| \bfc \|^2_{M^n} = \| \bfd \|^2_{\Lambda^n} =
\sum_{i=0}^n \lambda^n_i |\bfd_i|^2.
\end{equation}

Now, we consider the $M^n$ norm of $\sqrt{M^n}$.  Again, for any $\bfc \in \mathbb{R}^{n+1}$, we have
\begin{equation}
\left\| \sqrt{M^n} \bfc \right\|^2_{M^n}
= \left( \sqrt{M^n} \bfc \right) M^n \sqrt{M^n} \bfc
= \bfc^T (M^n)^2 \bfc
= \bfc^T Q^n \Lambda^2 (Q^n)^T \bfc,
\end{equation}
so that with $\bfd = (Q^n)^T \bfc$, we have
\begin{equation}
\left\| \sqrt{M^n} \bfc \right\|^2_{M^n}
= \bfd^T (\Lambda^n)^2 \bfd
= \sum_{i=0}^n (\lambda^n_i)^2 |\bfd_i|^2
\leq \lambda^n_{\max} \sum_{i=0}^n \lambda^n_i |\bfd_i|^2
= \lambda^n_{\max} \| \bfd \|^2_{\Lambda^n}.
\end{equation}

Consequently,
\begin{equation}
\left\| \sqrt{M^n} \right\|_{M^n}
= \max_{\| \bfc \|_{M^n} = 1} \left\| \sqrt{M^n} \bfc \right\|_{M^n}
= \sqrt{\lambda^n_{\max}}.
\end{equation}
\end{proof}

\begin{theorem}
The condition number of solving $M^n \bfc = \bfb$ measuring $\bfc$ in the $M^n$-norm and $\bfb$ in the 2-norm satisfies
\begin{equation}
\kappa_{M^n\rightarrow 2}(M^n) =
\sqrt{\kappa_2(M^n)} 
= \sqrt{\frac{(2n+1)!}{(n+1)! n!}}
\end{equation}
\end{theorem}

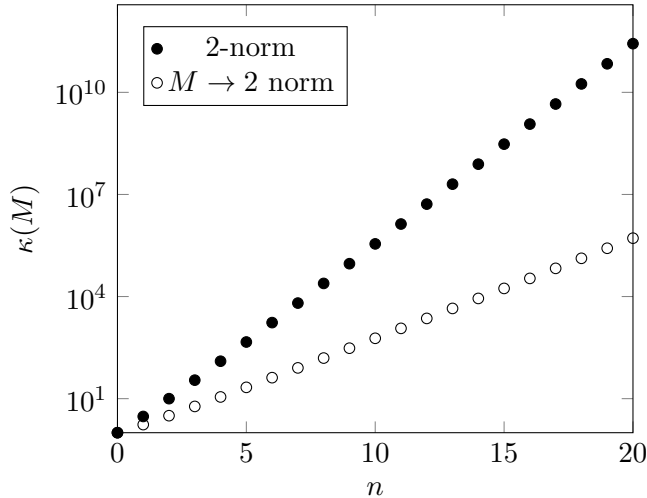
\begin{figure}
\begin{center}
\begin{tikzpicture}
\begin{semilogyaxis}[legend style={at={(0.05,0.95)},anchor=north west}, xlabel=$n$, ylabel=$\kappa(M)$, xmin=0, xmax=20, ymin=1]
\addplot [only marks, mark=*] table [x=n, y=kappa2, col sep=comma] {conds.dat};
\addplot [only marks, mark=o] table [x=n, y=kappam2, col sep=comma] {conds.dat};
\legend{2-norm,$M\rightarrow 2$ norm}
\end{semilogyaxis}
\end{tikzpicture}
\end{center}
\caption{Bernstein mass matrix conditioning for degrees 0 through 20 in the 2-norm and the $M\rightarrow 2$ norms.}
\label{fig:kappa}
\end{figure}

A plot of the condition numbers up to degree 20 in both norms is shown in Figure~\ref{fig:kappa}.  This discussion indicates that, when measuring the relevant $L^2$ rather than the Euclidean norm of the solution process, we can expect much better results than the alarming condition number in~\eqref{eq:kappa2} suggests.  It is important to note that nothing in this discussion, other than the eigenvalues of $M^n$, is particular to the univariate mass matrix.  Consequently, this discussion can inform the accuracy of the multivariate mass inversion process in~\cite{kirby2017fast} as well as the preconditioners for global mass matrices given in~\cite{AINSWORTH2019766}.

\section{Numerical results}
Now, we consider the accuracy of the methods described above on a range of problems.  We consider the best $L^2$ approximation of two smooth functions, $f(x) = 1/(1+396(x-0.5)^2)$ and $f(x) = 0.01 + x/(x^2+1)$ -- see Figure~\ref{fig:fplots} for plots of these.  Both functions are smooth.  However, the first function has large derivatives and so requires high order of approximation to obtain small error.  The second is not symmetric about $x=0.5$ but is otherwise relatively easy to approximate with a polynomial.

\begin{figure}
\centering
\begin{subfigure}{0.48\textwidth}
\centering
\begin{tikzpicture}[scale=0.85]
\begin{axis}[xlabel={$x$}, ylabel=$y$, samples=1000, thick, xmin=0, xmax=1, ymin=0, ymax=1, no marks]
\addplot[black]{1/(1+396*(x-0.5)^2)};
\end{axis}
\end{tikzpicture}
\caption{$1/(1+396(x-0.5)^2)$}
\label{fhard}
\end{subfigure}
\begin{subfigure}{0.48\textwidth}
\centering
\begin{tikzpicture}[scale=0.85]
\begin{axis}[xlabel={$x$}, ylabel={$y$}, samples=1000, thick, xmin=0, xmax=1, ymin=0, ymax=1.0, no marks]
\addplot[black]{1/100+x/(x^2+1)};
\end{axis}
\end{tikzpicture}
\caption{$0.01 + x/(1+x^2)$}
\label{fnonsymm}
\end{subfigure}
\caption{Plots of two functions being approximating with Bernstein polynomials.  Figure~\ref{fhard} is smooth, but has large derivatives and so needs high polynomial order to make the error small.  The function in Figure~\ref{fnonsymm} is much simpler to approximate, but illustrates that the methods work on functions not symmetric about the interval midpoint.}
\label{fig:fplots}
\end{figure}
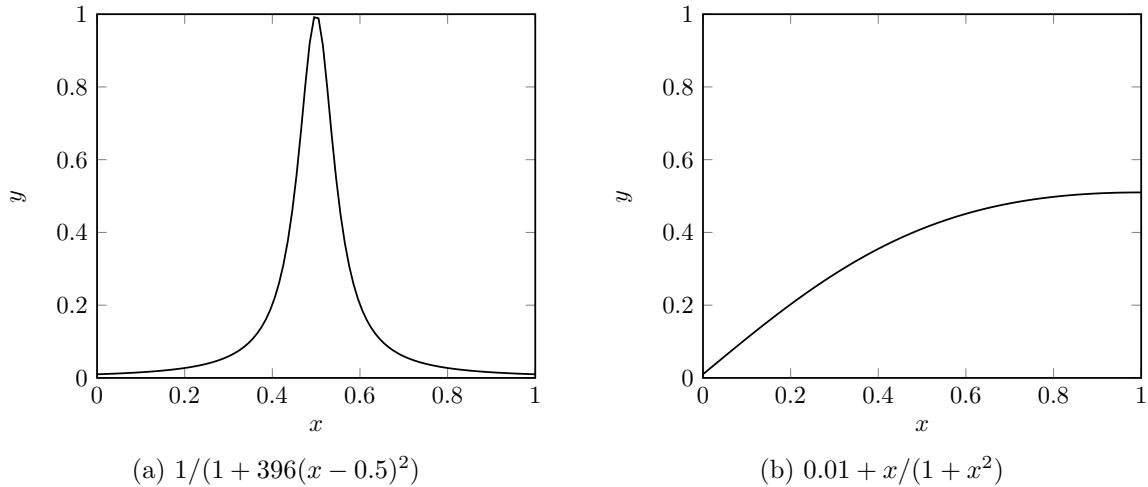

Solving~\eqref{eq:theeq} accurately, the best polynomial approximation of degree $n$ gives an $L^2$ error decreasing exponentially as a function of $n$.  To demonstrate this, Figures~\ref{subfig:ratfp} and~\ref{subfig:nonsymfp} show the $L^2$ error between the two functions and the $L^2$ projection computed with each of the four methods described above.  However, at various degrees, each of the methods deviate from yielding exponential convergence as roundoff error accumulates.  We note that the DFT-based method seems to be the worst, followed by multiplication by the exactly computed inverse.

We also compute the best approximation using the Legendre polynomials (which only involves numerical integration) and compute the $L^2$ difference between that and the polynomial produced by each solution technique.  Equivalently, this is the relative $M^n$-norm error between the exact and computed solution to~\eqref{eq:theeq}.  These plots appear in Figures~\ref{subfig:ratPifp} and~\ref{subfig:nonsymPifp}, futher demonstrating the instability of the DFT-based approach. 

Differences between the three more stable methods begin to appear when we consider the Euclidean norm of the error (Figures~\ref{subfig:raterr} and~\ref{subfig:nonsymerr}) and residual (Figures~\ref{subfig:ratres} and~\ref{subfig:nonsymres}) for each of our solution methods.  The DFT-based approach again performs much worse than the other methods, although the exact inverse gives 2-norm error comparable to the spectral and Cholesky approaches.  The residual plots show that the latter two methods give very small residual errors, not much above machine precision.  This strong backward stability of the Cholesky decomposition is a known property~\cite{golub2012matrix}, and we suspect (but have not proven) that it holds for our spectral decomposition as well.

To illustrate that these properties of the solution algorithm do not seem to depend on approximating a smooth solution, we also chose random solution vectors, computed the right-hand side by matrix multiplication, and then applied our four solution algorithms to attempt to recover the solution.  In Figure~\ref{fig:randominfo}, we see exactly the same behavior -- the DFT-based method behaving very badly, the Cholesky and spectral methods seemingly backward stable, and the exact matrix inverse somewhere in between.

These numerical results highlight several points.  First, although finite-dimensional norms are equivalent, the bounding constants can dramatically vary as a function of the matrix size.  Consequently, we can see comparable Euclidean norm errors but very different $M^n$-norm errors for, say the matrix inverse and Cholesky methods.  Second, the ill conditioning of our method is real, although we in fact see that the $M^n$ norm errors (Figures~\ref{subfig:ratPifp} and~\ref{subfig:nonsymPifp}) in our solution process for the (at least empirically) backward-stable methods are in fact quite a bit lower than the 2-norm errors as predicted by the conditioning analysis in Section~\ref{sec:cond}.  It seems that the empirical performance of the spectral method and Cholesky factorization are the best, although they have slightly different associated costs.  The spectral decomposition can be computed more quickly ($\mathcal{O}(n^2)$ versus $\mathcal{O}(n^3)$ for Cholesky), but the per-solve costs is higher -- two dense matrix multiplications rather than two triangular solves.

\begin{figure}
\centering
\vskip\baselineskip
\begin{subfigure}{0.475\linewidth}
\centering
\begin{tikzpicture}[scale=0.85]
\begin{semilogyaxis}[legend style={at={(0.05,0.05)}, anchor=south west}, xlabel=$n$, ylabel=$\| f-p \| / \| f \|$, xmin=0, xmax=20, ymin=1e-18]
\addplot [mark=square] table [x=n, y=directfp, col sep=comma] {approx_info1.dat};
\addplot [mark=triangle] table [x=n, y=DFTfp, col sep=comma] {approx_info1.dat};
\addplot [mark=diamond] table [x=n, y=Eigfp, col sep=comma] {approx_info1.dat};
\addplot[mark=o] table [x=n, y=chofp, col sep=comma]{approx_info1.dat};
\legend{$M^{-1}$, DFT, $Q\Lambda^{-1}Q^T$, $LL^T$}
\end{semilogyaxis}
\end{tikzpicture}
\caption{$L^2$ error between $f$ and $p$.}
\label{subfig:ratfp}
\end{subfigure}
\hfill
\begin{subfigure}{0.475\linewidth}
\centering
\begin{tikzpicture}[scale=0.85]
\begin{semilogyaxis}[legend style={at={(0.05,0.95)}, anchor=north west}, xlabel=$n$, ylabel=$\| \Pi f-p \| / \| f \|$, xmin=0, xmax=20, ymin=1e-18]
\addplot [mark=square] table [x=n, y=directPifp, col sep=comma] {approx_info1.dat};
\addplot [mark=triangle] table [x=n, y=DFTPifp, col sep=comma] {approx_info1.dat};
\addplot [mark=diamond] table [x=n, y=EigPifp, col sep=comma] {approx_info1.dat};
\addplot[mark=o] table [x=n, y=choPifp, col sep=comma]{approx_info1.dat};
\legend{$M^{-1}$, DFT, $Q\Lambda^{-1}Q^T$, $LL^T$}
\end{semilogyaxis}
\end{tikzpicture}
\caption{$L^2$ error between $\Pi f$ and $p$.}
\label{subfig:ratPifp}
\end{subfigure}
\vskip\baselineskip
\begin{subfigure}{0.475\linewidth}
\centering
\begin{tikzpicture}[scale=0.85]
\begin{semilogyaxis}[legend style={at={(0.05,0.95)}, anchor=north west}, xlabel=$n$, ylabel=$\| \bfx-\widehat{\bfx} \| / \| \bfx\|$, xmin=0, xmax=20, ymin=1e-18]
\addplot [mark=square] table [x=n, y=directerr, col sep=comma] {approx_info1.dat};
\addplot [mark=triangle] table [x=n, y=DFTerr, col sep=comma] {approx_info1.dat};
\addplot [mark=diamond] table [x=n, y=Eigerr, col sep=comma] {approx_info1.dat};
\addplot[mark=o] table [x=n, y=choerr, col sep=comma]{approx_info1.dat};
\legend{$M^{-1}$, DFT, $Q\Lambda^{-1}Q^T$, $LL^T$}
\end{semilogyaxis}
\end{tikzpicture}
\caption{Error in the 2-norm.}
\label{subfig:raterr}
\end{subfigure}
\hfill
\begin{subfigure}{0.475\linewidth}
\centering
\begin{tikzpicture}[scale=0.85]
\begin{semilogyaxis}[legend style={at={(0.05,0.95)}, anchor=north west}, xlabel=$n$, ylabel=$\| M\widehat{\bfx}-\bfb \| / \| \bfb \|$, xmin=0, xmax=20, ymin=1e-18]
\addplot [mark=square] table [x=n, y=directres, col sep=comma] {approx_info1.dat};
\addplot [mark=triangle] table [x=n, y=DFTres, col sep=comma] {approx_info1.dat};
\addplot [mark=diamond] table [x=n, y=Eigres, col sep=comma] {approx_info1.dat};
\addplot[mark=o] table [x=n, y=chores, col sep=comma]{approx_info1.dat};
\legend{$M^{-1}$, DFT, $Q\Lambda^{-1}Q^T$, $LL^T$}
\end{semilogyaxis}
\end{tikzpicture}
\caption{Residual in the 2-norm.}
\label{subfig:ratres}
\end{subfigure}
\caption{Relative error/residual in using the methods described in Section~\ref{sec:apply} to compute the degree $n$ approximation of $f(x)=\frac{1}{1+396(x-0.5)^2}$. $M^{-1}$ refers to the exact inverse, DFT refers to the factored inverse, $Q\Lambda^{-1}Q^T$ refers to the spectral decomposition, and $LL^T$ refers to the Cholesky factorization. We use $p$ to denote the computed approximation, $\Pi f$ to denote the best approximation, and $\widehat{\bfx}$ and $\bfx$ to denote their respective Bernstein coefficients. The vector $\bfb$ is given by $\bfb_i = \left( f(x), B^n_i(x)\right)_{L^2}$.}
\label{fig:ratinfo}
\end{figure}
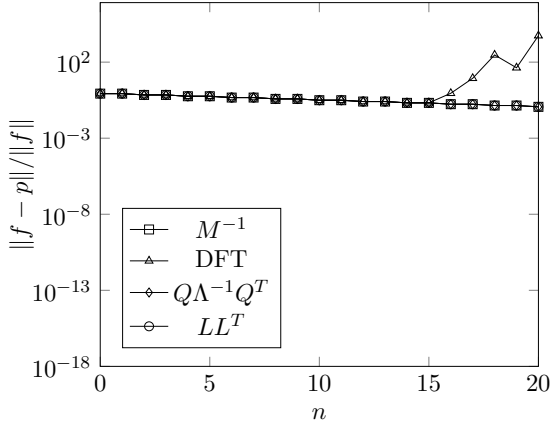
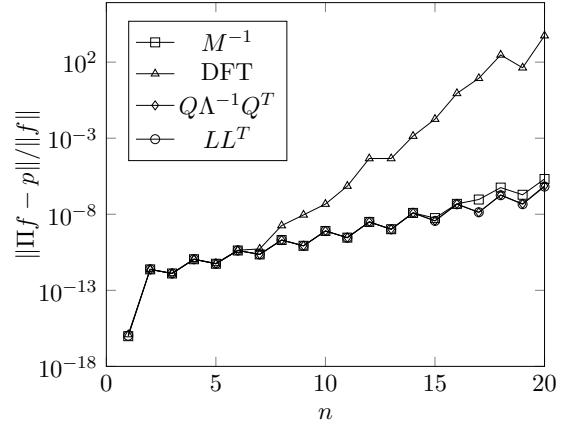
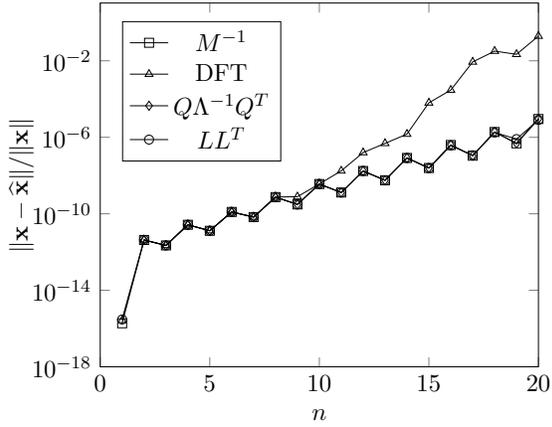
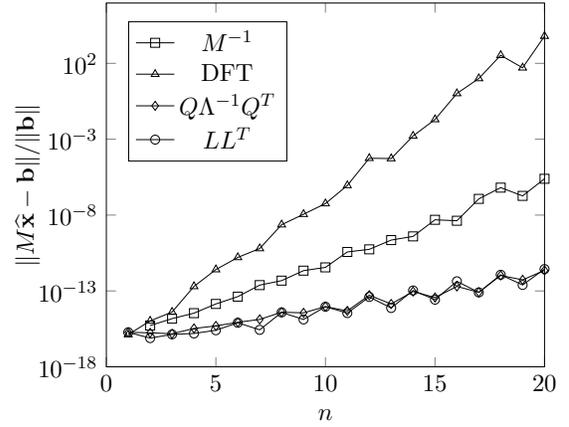

\begin{figure}
\centering
\vskip\baselineskip
\begin{subfigure}{0.475\linewidth}
\centering
\begin{tikzpicture}[scale=0.85]
\begin{semilogyaxis}[legend style={at={(0.05,0.05)}, anchor=south west}, xlabel=$n$, ylabel=$\| f-p \| / \| f \|$, xmin=0, xmax=20, ymin=1e-18]
\addplot [mark=square] table [x=n, y=directfp, col sep=comma] {approx_info2.dat};
\addplot [mark=triangle] table [x=n, y=DFTfp, col sep=comma] {approx_info2.dat};
\addplot [mark=diamond] table [x=n, y=Eigfp, col sep=comma] {approx_info2.dat};
\addplot[mark=o] table [x=n, y=chofp, col sep=comma]{approx_info2.dat};
\legend{$M^{-1}$, DFT, $Q\Lambda^{-1}Q^T$, $LL^T$}
\end{semilogyaxis}
\end{tikzpicture}
\caption{$L^2$ error between $f$ and $p$.}
\label{subfig:nonsymfp}
\end{subfigure}
\hfill
\begin{subfigure}{0.475\linewidth}
\centering
\begin{tikzpicture}[scale=0.85]
\begin{semilogyaxis}[legend style={at={(0.05,0.95)}, anchor=north west}, xlabel=$n$, ylabel=$\| \Pi f-p \| / \| f \|$, xmin=0, xmax=20, ymin=1e-18]
\addplot [mark=square] table [x=n, y=directPifp, col sep=comma] {approx_info2.dat};
\addplot [mark=triangle] table [x=n, y=DFTPifp, col sep=comma] {approx_info2.dat};
\addplot [mark=diamond] table [x=n, y=EigPifp, col sep=comma] {approx_info2.dat};
\addplot[mark=o] table [x=n, y=choPifp, col sep=comma]{approx_info2.dat};
\legend{$M^{-1}$, DFT, $Q\Lambda^{-1}Q^T$, $LL^T$}
\end{semilogyaxis}
\end{tikzpicture}
\caption{$L^2$ error between $\Pi f$ and $p$.}
\label{subfig:nonsymPifp}
\end{subfigure}
\vskip\baselineskip
\begin{subfigure}{0.475\linewidth}
\centering
\begin{tikzpicture}[scale=0.85]
\begin{semilogyaxis}[legend style={at={(0.05,0.95)}, anchor=north west}, xlabel=$n$, ylabel=$\| \bfx-\widehat{\bfx} \| / \| \bfx\|$, xmin=0, xmax=20, ymin=1e-18]
\addplot [mark=square] table [x=n, y=directerr, col sep=comma] {approx_info2.dat};
\addplot [mark=triangle] table [x=n, y=DFTerr, col sep=comma] {approx_info2.dat};
\addplot [mark=diamond] table [x=n, y=Eigerr, col sep=comma] {approx_info2.dat};
\addplot[mark=o] table [x=n, y=choerr, col sep=comma]{approx_info2.dat};
\legend{$M^{-1}$, DFT, $Q\Lambda^{-1}Q^T$, $LL^T$}
\end{semilogyaxis}
\end{tikzpicture}
\caption{Error in the 2-norm.}
\label{subfig:nonsymerr}
\end{subfigure}
\hfill
\begin{subfigure}{0.475\linewidth}
\centering
\begin{tikzpicture}[scale=0.85]
\begin{semilogyaxis}[legend style={at={(0.05,0.95)}, anchor=north west}, xlabel=$n$, ylabel=$\| M\widehat{\bfx}-\bfb \| / \| \bfb \|$, xmin=0, xmax=20, ymin=1e-18]
\addplot [mark=square] table [x=n, y=directres, col sep=comma] {approx_info2.dat};
\addplot [mark=triangle] table [x=n, y=DFTres, col sep=comma] {approx_info2.dat};
\addplot [mark=diamond] table [x=n, y=Eigres, col sep=comma] {approx_info2.dat};
\addplot[mark=o] table [x=n, y=chores, col sep=comma]{approx_info2.dat};
\legend{$M^{-1}$, DFT, $Q\Lambda^{-1}Q^T$, $LL^T$}
\end{semilogyaxis}
\end{tikzpicture}
\caption{Residual in the 2-norm.}
\label{subfig:nonsymres}
\end{subfigure}
\caption{Relative error/residual in using the methods described in Section~\ref{sec:apply} to compute the degree $n$ approximation of $f(x)=0.01+\frac{x}{x^2+1}$. $M^{-1}$ refers to the exact inverse, DFT refers to the factored inverse, $Q\Lambda^{-1}Q^T$ refers to the spectral decomposition, and $LL^T$ refers to the Cholesky factorization. We use $p$ to denote the computed approximation, $\Pi f$ to denote the best approximation, and $\widehat{\bfx}$ and $\bfx$ to denote their respective Bernstein coefficients. The vector $\bfb$ is given by $\bfb_i = \left( f(x), B^n_i(x)\right)_{L^2}$.}
\label{fig:nonsyminfo}
\end{figure}
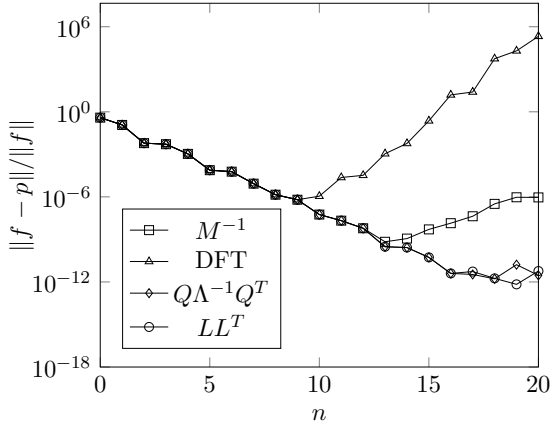
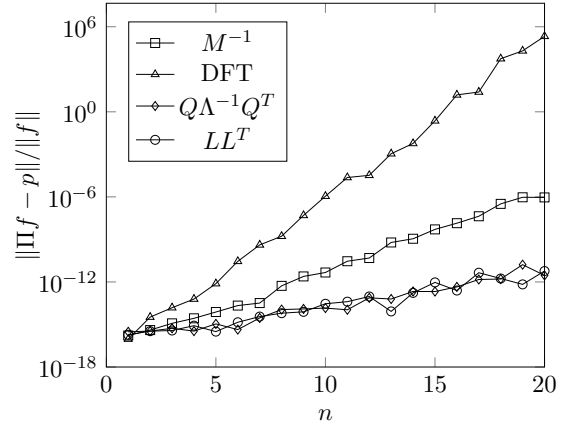
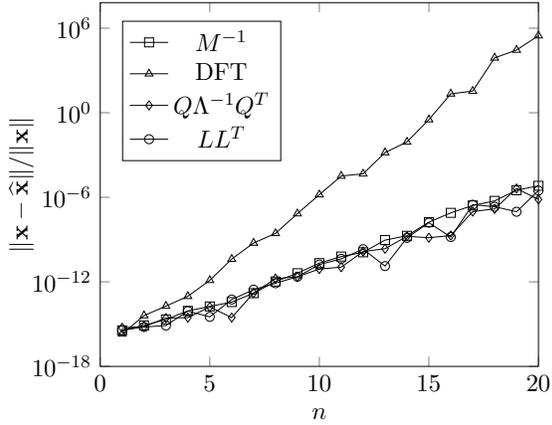
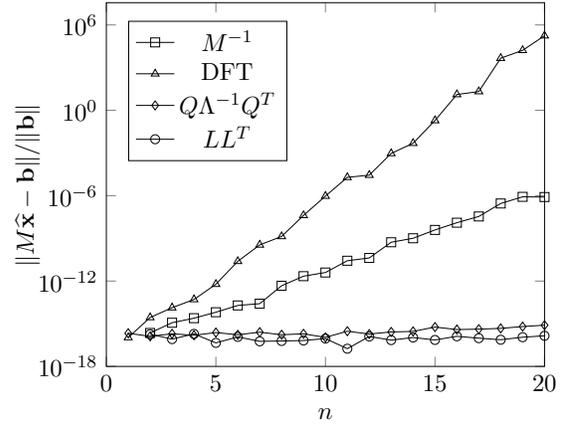

\begin{figure}
\centering
\begin{subfigure}{0.475\linewidth}
\centering
\begin{tikzpicture}[scale=0.85]
\begin{semilogyaxis}[legend style={at={(0.05,0.95)}, anchor=north west}, xlabel=$n$, ylabel=$\| \bfx-\widehat{\bfx} \|_2 / \| \bfx \|_2$, xmin=0, xmax=20, ymin=1e-18]
\addplot [mark=square] table [x=n, y=directL2err, col sep=comma] {random_info.dat};
\addplot [mark=triangle] table [x=n, y=DFTL2err, col sep=comma] {random_info.dat};
\addplot [mark=diamond] table [x=n, y=EigL2err, col sep=comma] {random_info.dat};
\addplot[mark=o] table [x=n, y=choL2err, col sep=comma]{random_info.dat};
\legend{$M^{-1}$, DFT, $Q\Lambda^{-1}Q^T$, $LL^T$}
\end{semilogyaxis}
\end{tikzpicture}
\caption{Error in the 2-norm.}
\label{subfig:2err}
\end{subfigure}
\hfill
\begin{subfigure}{0.475\linewidth}
\centering
\begin{tikzpicture}[scale=0.85]
\begin{semilogyaxis}[legend style={at={(0.05,0.95)}, anchor=north west}, xlabel=$n$, ylabel=$\| \bfx-\widehat{\bfx} \|_M / \| \bfx \|_M$, xmin=0, xmax=20, ymin=1e-18]
\addplot [mark=square] table [x=n, y=directMerr, col sep=comma] {random_info.dat};
\addplot [mark=triangle] table [x=n, y=DFTMerr, col sep=comma] {random_info.dat};
\addplot [mark=diamond] table [x=n, y=EigMerr, col sep=comma] {random_info.dat};
\addplot[mark=o] table [x=n, y=choMerr, col sep=comma]{random_info.dat};
\legend{$M^{-1}$, DFT, $Q\Lambda^{-1}Q^T$, $LL^T$}
\end{semilogyaxis}
\end{tikzpicture}
\caption{Error in the $M$ norm.}
\label{subfig:Merr}
\end{subfigure}
\vskip\baselineskip
\begin{subfigure}{\linewidth}
\centering
\begin{tikzpicture}[scale=0.85]
\begin{semilogyaxis}[legend style={at={(0.05,0.95)}, anchor=north west}, xlabel=$n$, ylabel=$\| M\widehat{\bfx}-\bfb \|_2$, xmin=0, xmax=20, ymin=1e-18]
\addplot [mark=square] table [x=n, y=directres, col sep=comma] {random_info.dat};
\addplot [mark=triangle] table [x=n, y=DFTres, col sep=comma] {random_info.dat};
\addplot [mark=diamond] table [x=n, y=Eigres, col sep=comma] {random_info.dat};
\addplot[mark=o] table [x=n, y=chores, col sep=comma]{random_info.dat};
\legend{$M^{-1}$, DFT, $Q\Lambda^{-1}Q^T$, $LL^T$}
\end{semilogyaxis}
\end{tikzpicture}
\caption{Residual in the 2-norm.}
\label{subfig:2res}
\end{subfigure}
\caption{Error/residual in using the methods described in Section~\ref{sec:apply} to solve $M\bfx=\bfb$, where $\bfb$ is a random vector in $[-0.5,0.5]^{n+1}$. $M^{-1}$ refers to the exact inverse, DFT refers to the factored inverse, $Q\Lambda^{-1}Q^T$ refers to the spectral decomposition, and $LL^T$ refers to the Cholesky factorization. We use $\widehat{\bfx}$ to denote the computed solution.}
\label{fig:randominfo}
\end{figure}
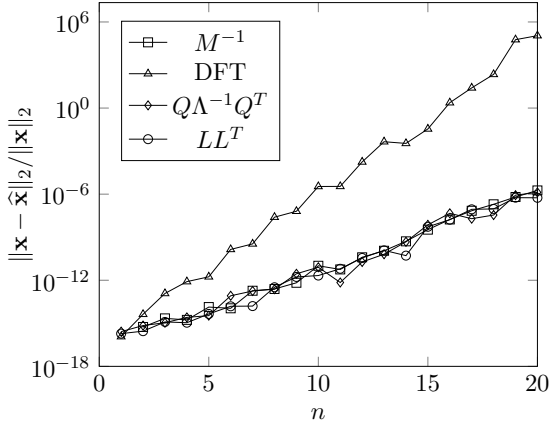
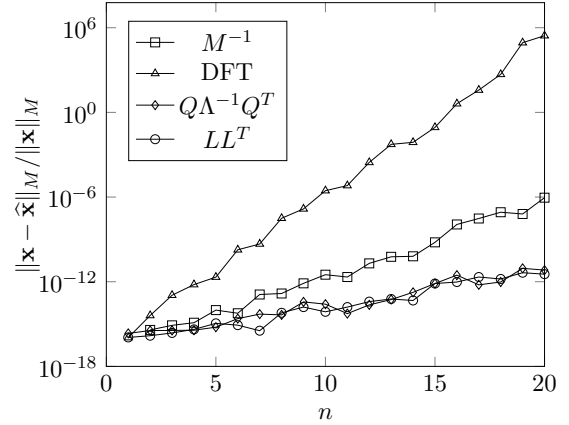
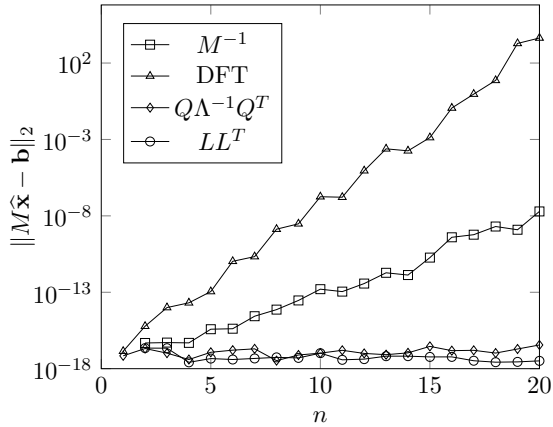

\section{Conclusions}
We have studied several algorithms for the inversion of the univariate Bernstein mass matrix.  Our fast algorithm for inversion based on the spectral decomposition seems very stable in practice and has accuracy comparable to the Cholesky decomposition, while a superfast algorithm based on the discrete Fourier transform is unfortunately unstable.  Moreover, we have given a new perspective on the conditioning of matrices for polynomial projection indicating that the problems are better-conditioned with respect to the $L^2$ norm of the output than the Euclidean norm.  In the future, we hope to expand this perspective to other polynomial problems and continue the development of fast and accurate methods for problems involving Bernstein polynomials.
\bibliographystyle{plain}
\bibliography{references.bib}

\end{document}